\documentclass[12pt,reqno]{amsart}
\usepackage{amssymb,verbatim}
\usepackage[active]{srcltx}
\usepackage{graphicx}
\usepackage{cite}
\usepackage{verbatim,amscd,amssymb}
\usepackage{graphicx}
\usepackage{amsmath}
\usepackage{inputenc}
\usepackage{enumitem}
\usepackage{float}
\usepackage[active]{srcltx}
\graphicspath{ {./images/} }
\usepackage{fontenc}
\usepackage{float}

\newtheorem{theorem}{Theorem}[section]
\newtheorem{lemma}[theorem]{Lemma}
\newtheorem{cor}[theorem]{Corollary}

\theoremstyle{definition}
\newtheorem{definition}[theorem]{Definition}
\newtheorem{example}[theorem]{Example}

\theoremstyle{remark}

\numberwithin{equation}{section}

\newcommand{\uc}{\mathbb{S}}

\newcommand{\sha}{\succ\mkern-14mu_s\;}

\begin{document}

\date{August 16, 2019}

\title[ VERY BADLY ORDERED CYCLES OF INTERVAL MAPS]
{ VERY BADLY ORDERED CYCLES OF INTERVAL MAPS}

\dedicatory{Dedicated to the occasion of Professor A. N. Sharkovsky 82-nd birthday}

\author{Sourav Bhattacharya}

\author{Alexander Blokh}

\address[Sourav Bhattacharaya and Alexander Blokh]
{Department of Mathematics\\ University of Alabama at Birmingham\\
	Birmingham, AL 35294}



\begin{abstract}

We prove that a periodic orbit $P$ with coprime over-rotation pair is
an over-twist periodic orbit iff the $P$-linear map has the over-rotation
interval with left endpoint equal to the over-rotation number
of $P$. We show that this fails if the over-rotation pair
of $P$ is not coprime. We give examples of patterns with non-coprime
over-rotation pairs, no block structure over over-twists, and with  over-rotation number equal to the
left endpoint of the forced over-rotation interval (call them
\emph{very badly ordered}, similar to patterns of degree one circle maps in
\cite{alm98}). This presents a situation in which the results about
over-rotation numbers on the interval and those about classical
rotation numbers for circle degree one maps are different. In the end we explicitly describe the strongest unimodal pattern that forces a given over-rotation interval and use it to construct unimodal very badly ordered patterns with arbitrary non-coprime over-rotation pairs.

\end{abstract}

\maketitle

\section*{Introduction }

One-dimensional dynamics is concerned with the asymptotic behavior of
sequences $\{x_n\}$ defined iteratively by $x_{n+1} = f(x_n)$ where $f$
is an arbitrary continuous map of an interval to itself. The sequence
$\{x_n\}$ is called the \emph{trajectory} of the \emph{initial} point
$x_0$ under the map $f$. An important reason for studying one
dimensional dynamics, in addition to its intrinsic interest, comes from
higher-dimensional dynamics.

If in the \emph{trajectory} $\{x_n \}$ we have $x_p = x_0$ for some
minimal $p>0$, then $x_{n+p}=x_n$ for every $n>0$. Thus, the trajectory
$\{x_n\}$ is \emph{periodic}; the trajectory and its initial point
$x_0$ are then said to be of \emph{period} $p$ (evidently, all points
of this trajectory are then of the same period), and the set $\{x_0,
\dots, x_{p-1}\}$ is then called a \emph{cycle} or a \emph{periodic
orbit (of the point $x_0$)}. It turns out that a continuous interval
map $f$ has periodic points of periods that are related to one another.
A complete description of all possible sets of periods of periodic
points of a continuous map of an interval is due to A. N. Sharkovsky.

The celebrated Sharkovsky Theorem \cite{S} illustrates the rigid
restrictions on the set of periods of periodic orbits of a continuous
interval map. To state it let us first introduce the \emph{Sharkovsky
ordering} for the set $\mathbb{N}$ of positive integers:
$$3\sha 5\sha 7\sha\dots\sha 2\cdot3\sha 2\cdot5\sha 2\cdot7 \sha \dots $$
$$\sha\dots 2^2\cdot3\sha 2^2\cdot5\sha 2^2\cdot7\sha\dots\sha 8\sha
4\sha 2\sha 1.$$
If $m\sha n$, say that $m$ is {\it sharper} than
$n$. Denote by $Sh(k)$ the set of all positive integers $m$ such that
$k\sha m$, together with $k$, and by $Sh(2^\infty)$ the set
$\{1,2,4,8,\dots\}$ which includes all powers of $2$. Denote also by
$P(f)$ the set of the periods of cycles of a map $f$. The following theorem
was proven by A. N. Sharkovsky.

\begin{theorem}[\cite{S}]\label{t:shar}
 	If $f:[0,1]\to [0,1]$ is a continuous map, $m\sha n$ and $m\in
 	P(f)$, then $n\in P(f)$. Therefore there exists $k \in \mathbb{N}
 	\cup \{2^\infty\}$ such that $P(f)=Sh(k)$. Conversely, if $k\in
 	\mathbb{N}\cup \{2^\infty\}$ then there exists a continuous map
 	$f:[0,1]\to [0,1]$ such that $P(f)=Sh(k)$. 		
\end{theorem}

The above theorem provides a full description of possible sets of
periods of cycles of continuous interval maps. Moreover, it shows that
various periods \emph{force} one another in the sense that if $m \sha
n$ then, for a continuous interval map $f$, the existence of a cycle of
period $m$ \emph{forces} the existence of a cycle of period $n$. While
the statement of Sharkovsky Theorem focuses on the periods of cycles of
interval maps, its various proofs indicate a rich combinatorial
structure controlling the disposition of the orbits themselves. This
led to an explosion of interest to interval maps and gave birth to the
field of Combinatorial One-Dimensional Dynamics (see \cite{alm00}).

However, the period is a rough characteristic of a cycle as there are a
lot of cycles of the same period. A much finer way of describing cycles
is by considering its \emph{(cyclic) permutation}, that is, the cyclic
permutation that we get when we look at how the map acts on the points
of the cycle, ordered from the left to the right. In what follows we
often identify cyclic permutations with families of all cycles on the
real line that induce such permutations and called \emph{patterns} and
use the two terms ``permutation'' and ``pattern'' interchangeably.
Patterns are partially ordered by the \emph{forcing relation}. A
pattern $A$ \emph{forces} pattern $B$ if every continuous map having a
cycle of pattern $A$ has a cycle of pattern $B$. By \cite{Ba} forcing
is a partial ordering. Thus, if we know which patterns are forced by a
given pattern $A$, we have enormous information about the structure of
an interval map with a cycle of pattern $A$.

A useful algorithm allows one to describe all patterns forced by a
given pattern $A$. Namely, consider a cycle $P$ of pattern $A$; assume
that the leftmost point of $P$ is $a$ and the rightmost point of $P$ is
$b$. Every component of $[a, b]\setminus P$ is said to be a
\emph{$P$-basic interval}. Extend the map from $P$ to the interval $[a,
b]$ by defining it linearly on each $P$-basic interval and call the
resulting map $f_P$ the \emph{$P$-linear map}. Then the patterns of all
cycles of $f_P$ are exactly the patterns forced by the pattern of $P$
(see \cite{Ba} and \cite{alm00}).

As it turns out, the forcing relation is rather complicated.  This
motivates one to look for another, middle-of-the-road way of describing
cycles, a way not as crude as periods but not as fine as permutations,
which would still allow for a transparent description.

In \cite{BM1} a new notion of a type of a cycle, the
\emph{over-rotation pair}, was introduced. If $P$ is a cycle of $f$ of
period $q$ , then the over-rotation pair of $P$ is $orp(P)=(p,q)$,
where $2p$ is the number of points $x\in P$ such that $f(x)-x$ and
$f^2(x) - f(x)$ have different signs (if $f$ has only one fixed point
$a$, then it is easy to check that $p$ is equal to the number of points
$x\in P$ such that $x>a$ and $f(x)<a$). The number
$\frac{p}{q}=\rho(P)$ is called the \emph{over-rotation number} of $P$.
Similarly one defines the over-rotation pair $orp(\pi)$ and number
$\rho(\pi)$ of a pattern $\pi$; in fact in what follows we will often
define combinatorial concepts for cycles assuming by default that the
same concept can be similarly defined for patterns. It turns out that
the forcing relation among over-rotation pairs is also linear (as for
periods) so that the family of all over-rotation pairs of a given
interval map can be easily described. On the other hand, one gets much
more information from the over-rotation pair than from the period
alone. Therefore, one may consider over-rotation pairs as a good
compromise between patterns and periods. We continue to investigate
them in the present paper.

%It is useful to be able to derive many consequences from a few bits of
%information.  This occurs particularly a lot in one-dimensional
%dynamics when knowing a pattern we can often say a lot about the
%system.

Namely, by \cite{BM1} if $(p, q)$ and $(r, s)$ are over-rotation pairs
such that $p/q<r/s$ then any interval map with cycle of over-rotation
pair $(p, q)$ must have a cycle of over-rotation pair $(r, s)$. Similar
to the case of periods one can say that $(p, q)$ \emph{forces} $(r,
s)$. The set of all over-rotation pairs of an interval map $f$ is
denoted $ORP(f)$; similar to how Theorem \ref{t:shar} implies a full
description of all possible sets of periods of cycles of $f$, results
of \cite{BM1} imply a full description of all possible sets $ORP(f)$
for continuous interval maps $f$. Observe that over-rotation pairs and
numbers are only defined for non-fixed periodic points; hence, if a
continuous interval map $f$ has no periodic non-fixed points, the
over-rotation pairs and numbers are not defined for such a map. On the
other hand, it is well-known that any such map $f$ has trivial
dynamics: for any point $x$ the limit set of $x$ is a fixed point. So,
from now on we study only maps with some non-fixed periodic points.

Observe that any over-rotation number is at most equal to
$\frac{1}{2}$. Moreover, if $f$ has non-fixed periodic points then by
Theorem \ref{t:shar} it has a point of period $2$. Thus, the
over-rotation pair $(1, 2)$ always belongs to $ORP(f)$. Hence by
\cite{BM1} the closure of the set of all over-rotation numbers of
cycles of $f$ is an interval $I_f=[r_f, 1/2]$ called the
\emph{over-rotation} interval of $f$. Given a pattern $\pi$, we say
that it forces the over-rotation interval $I_\pi=[r_\pi, 1/2]$ defined
as the over-rotation interval of the $P$-linear map $f_P$ where $P$ is
a cycle of pattern $\pi$. It follows from the above description of the
patterns forced by a given pattern that $I_\pi$ is well-defined.
Moreover, by \cite{BM1} we have $r_\pi \le \rho(\pi)\le 1/2$. In the
present paper we study the ``extreme'' case when $r_\pi=\rho(\pi)$.

Results of \cite{BM1} imply Theorem~\ref{t:shar}. Indeed, let $f$ be an
interval map and consider odd periods. For any $2n+1$ the closest to
$1/2$ over-rotation number of a periodic point of period $2n+1$ is
$\frac{n}{2n+1}$. Clearly $\frac{n}{2n+1}<\frac{n+1}{2n+3}<\frac12$.
%and $\frac {n+1}{2n+1}>\frac{n+2}{2n+3}>\frac12$.
Hence for any periodic point $x$ of period $2n+1$ its over-rotation
pair
$orp(x)$ is $\gtrdot$-stronger than the pair $(n+1,2n+3)$, % or the
%pair $(n+2,2n+3)$;
and by \cite{BM1} the map $f$ has a point of period $2n+3$. Also, for
any $m$ we have $(n,2n+1)\gtrdot (m,2m)$, so by \cite{BM1} the
map $f$ has a point of any even period $2m$. Applying this to the maps
$f,f^2,f^4,\dots$ one can prove Theorem~\ref{t:shar} for all periods
but the powers of $2$; additional arguments covering the case of powers
of $2$ are quite easy.

Let us now describe our plans in more detail. Forcing-minimal patterns
among patterns of a given over-rotation number are called
\emph{over-twist patterns} or just \emph{over-twists}. By \cite{BM1},
an over-twist pattern has a coprime over-rotation pair; in particular,
over-twists of over-rotation number $1/2$ are of period $2$, so from
now on we consider over-twists of over-rotation numbers distinct from
$1/2$. By \cite{BM1} and properties of the forcing relation, if an
interval map $f$ has the over-rotation interval $I_f=[r_f, 1/2]$, then
for any $p/q\in (r_f, 1/2]$ there exists a cycle of $f$ that exhibits
an over-twist pattern of over-rotation number $p/q$. Loosely speaking,
over-twists are patterns that are \emph{guaranteed} to a map $f$ if its
over-rotation interval contains the appropriate over-rotation number.
Moreover, again by \cite{BM1}, if $\pi$ is an over-twist pattern then
it forces the over-rotation interval $[\rho(\pi), 1/2]$ (an over-twist
cannot force a pattern of smaller over-rotation number). So, for
over-twists the desired equality $r_\pi=\rho(\pi)$ holds. We prove a
version of the opposite statement too; more precisely, we prove that a
pattern $\pi$ with \emph{coprime} over-rotation pair (i.e.,
over-rotation pair $(p, q)$  such that $p$ and $q$ are coprime) is an
over-twist \emph{if and only if} $r_\pi=\rho(\pi)$ (the ``only if''
direction has been discussed above) and then study the case of
non-coprime over-rotation pairs.

In fact, there is a natural class of patterns with non-coprime
over-rotation  pairs related to over-twists for whom the same holds. Set $T_n=\{1, 2, \dots, n\}$. A
permutation $\pi$ is said to have a {\it block structure} if there is a
collection of pairwise disjoint segments $I_0, \dots, I_k$ with
$\pi(T_n\cap I_j)=T_n\cap I_{j+1}, \pi(T_n\cap I_k)=T_n\cap I_0$; the
intersections of $T_n$ with intervals $I_j$ are called {\it blocks} of
$\pi$. If we collapse blocks to points, we get a new permutation
$\pi'$, and then $\pi$ is said to have a block structure {\it over
$\pi'$}. A permutation without a block structure is said to have
\emph{no block structure}, or, equivalently, to be {\it irreducible}.
If $\pi$ has a block structure over a pattern $\theta$, then $\pi$
forces $\theta$. By \cite{MN} for each permutation $\pi$ there exists a
unique irreducible pattern $\pi'$ over which $\pi$ has block structure
(thus, $\pi'$ is forced by $\pi$).

Coming back to over-rotation numbers, observe that if a permutation
$\pi$ has a block structure over a non-degenerate pattern $\zeta$ then
it is easy to see that $\rho(\pi)=\rho(\zeta)$ and $r_\pi=r_\zeta$.
Hence if $\zeta$ is an over-twist then $\pi$ with block structure over
$\zeta$ also has the desired property $\rho(\pi)=r_\pi$.

In the present paper we study \emph{other} patterns $\pi$ such that
$I_\pi=[\rho(\pi), 1/2]$; in other words, we want patterns $\pi$ that
force only patterns of equal or greater over-rotation numbers but do
\emph{not} have block structure over over-twists of the same
over-rotation number. Surprisingly, it turns out that there exist such
patterns; this provides an example when the situation for the
over-rotation numbers (defined on the interval) the situation is
different from that for degree one circle maps.

Thus, we show that the implication ``if $\pi$ is a pattern such that
$r_\pi=\rho(\pi)$ then $\pi$ must have a block structure over an
over-twist pattern'' fails if the over-rotation pair of $\pi$ is not
coprime. Call a pattern $\pi$ \emph{badly ordered} if it does not have a block structure over an over-twist pattern of over-rotation number $\rho(\pi)$. This is analogous to the badly ordered patterns of degree one circle maps defined as patterns that do not have block structure over circle rotations. In this paper we first develop a verifiable computational  criterion for the property $r_\pi=\rho(\pi)$ to hold. Then
we discover irreducible patterns $\pi$ with
non-coprime over-rotation pairs such that $r_\pi=\rho(\pi)$ (observe
that since the over-rotation pair of $\pi$ is not coprime, $\pi$ is not
an over-twist). Such patterns are called \emph{very badly ordered} for the following reason. In \cite{alm98} \emph{badly ordered} cycles of degree one circle maps are defined as ones of rotation number, say, $p/q$ but without a block structure over a circle rotation by the angle $p/q$; it is shown in \cite{alm98} that such cycles force the rotation interval containing $p/q$ in in its interior. Very badly ordered interval cycles show that on the interval this is not true. We provide a simple algorithm to construct unimodal badly ordered patterns with arbitrary over-rotation pair.

Observe, that very badly ordered interval patterns present a surprising
departure from the previously observed phenomenon according to which
the results about over-rotation numbers on the interval and those about
classical rotation numbers for circle maps of degree one are analogous.
As a part of our study we elucidate a description of the strongest
unimodal pattern that corresponds to a given over-rotation interval.

Our paper is divided into sections as follows:

\begin{enumerate}

\item Section 1 contains preliminaries.

\item In Section 2 we establish some properties of the concept of the code (this concept was introduced in Section 1).
 	
\item In Section 3 we define \emph{very badly ordered cycles} and describe properties of their code.

\item In Section 4 the kneading sequences of unimodal maps are studied. In its first subsection we describe the strongest kneading sequence associated with a given over-rotation interval with rational left endpoint. In the second subsection we give an explicit description of maps discussed in the previous subsection.

\item In Section 5 we construct unimodal very badly ordered cycles with an arbitrary non-coprime over-rotation pair.

	\end{enumerate}

\section{ Preliminaries }\label{s:prelim}

We need some known facts; references given in each claim are not
unique. Let $I_0,\dots$ be closed intervals such that $f(I_j) \supset
I_{j+1}$ for $ j \geq 0 $; then we say that $I_0, \dots $ is an
\emph{$f-$chain} or simply a \emph{chain} of intervals. If a finite
chain of intervals $I_0, I_1, \dots , I_{k-1}$ is such that $f(I_{k-1})
\supset I_0 $, then we call $I_0, I_1, \dots , I_{k-1}$ an
$f$-\emph{loop} or simply a \emph{loop} of intervals.

\begin{lemma}[\cite{alm00}]\label{ALM2}
The following statements are true.
\begin{enumerate}
	\item Let $I_0, \dots, I_k $ be a finite chain of intervals.
Then there  is an interval $M_k$ such that $ f^j(M_j) \subset I_j $
for $ 0 \leq j \leq k-1 $ and $ f^k(M_k) = I_k.$
	\item Let $I_0, \dots $ be an infinite chain of intervals. Then
there is a nested sequence of intervals $M_k$ defined as in $(1)$
whose intersection is an interval $M$ such that $f^j(M) \subset I_j
$ for all $j$. 	
	\item Let $I_0, \dots, I_{k-1}$ be a loop of intervals. Then
there is a periodic point $x$ such that $f^j(x) \in I_j $ for $
0\leq j \leq k-1 $ and $ f^k(x)=x$
\end{enumerate}
\end{lemma}

Let $\mathcal{U}$ be the set of all piecewise-monotone interval maps
$g$ with one fixed point (denoted $a_g = a$). Fix $ f \in \mathcal{U}$,
then $f(x)> x$ for any $x < a$ and $f(x) < x $ for any $x>a$. We call
an interval $I$  \emph{admissible} if one of its endpoints is $a$. We
call a  chain(a loop) of admissible intervals $I_0,I_1,\dots $ an
\emph{admissible loop(chain)} respectively. For any admissible loop
$\overline{\alpha} = \{ I_0, I_1, \dots , I_{k-1} \} $ , we call the
pair of numbers $(\frac{p}{2}, k)=orp(\overline{\alpha})$, the
\emph{over-rotation pair} of $\overline{\alpha}$ where $p$ is the
number of subscripts $s, 0\leq s\leq k-1,$ with $I_s$ and $I_{s+1 \mod
k}$ located on opposite sides of $a$.

Evidently, this definition is consistent with the definition of
over-rotation pair given above. Observe, that for an admissible loop
the number $p$ is always even (as the interval has to come back to
where the loop starts). Call the number $ \rho(\overline{\alpha})=
\frac{p}{2k}$ the \emph{over-rotation number} of $\overline{\alpha}$. A
sequence $ \{y_1, y_2, \dots, y_l \}$ is \emph{non-repetitive} if it
cannot be represented as several repetitions of a shorter sequence.
Define a function $ \varphi_a $ on all points not equal to $a$ as follows:
$\varphi_a(b)=0$ if $b<a$ and $\varphi_a(d) =1$ if $a<d$. Also, set
$\varphi_a([b,a])=0$ if $b<a$ and $\varphi_a([a,d]) =1$ if $a<d$. Given
a set $A$ and a point $x$, write $A \leq x $ if for any $y \in A $ we
have $y \leq x $. If $x<y<a$ or $x>y>a$, we write $x>_a y$; if $x\le
y<a$ or $x\ge y>a$, we write $x\ge_a y$.

\begin{lemma}[\cite{BM1}]\label{BM3}
Let $f$ $\in \mathcal{U}$ and let $ \overline{\alpha} = \{I_0, I_1,
\dots , I_{k-1} \} $ be an admissible loop of non-degenerate intervals.
Then there are the following possibilities.

\begin{enumerate}

\item The number $k$ be even, for each $j$ the intervals $I_j$ and
    $I_{j+1}$ are such that either $I_j \leq a \leq I_{j+1}$ or $
    I_j \geq a \geq I_{j+1}$, and the map $f$ has a point $x$ of
    period 2.

\item If the above possibility fails, then there is a periodic
    point $x\in I_0$ such that $ x\neq a $, $f^j(x)\in I_j (0 \leq
    j \leq k-1)$, $f^k(x)=x$, and $\rho(x) =
    \rho(\overline{\alpha})$. If the sequence $\{ \varphi_a(I_0),
    \varphi_a(I_1), \dots , \varphi_a(I_{k-1}) \}$ is
    non-repetitive , then $orp(x)=orp(\overline{\alpha})$.
    Moreover, $x$ can be chosen so that the following holds: for
    every $y$ from the orbit of $x$ there exists no $z$ such that
    $y>_a z$ and $f(y) = f(z)$.

\end{enumerate}

\end{lemma}

A point $x$ from Lemma \ref{BM3}(2) is said to be \emph{generated} by $
\overline{\alpha}$.

Let us state results of \cite{BM2}. One of them gives a criterion for a
pattern to be an over-twist pattern. To state it we need to define the
\emph{code}, that is, a special function which maps points of either
periodic orbit or of a pattern to the reals. We also need a few other
definitions.

A cycle $P$ of a map $f$, and the pattern it represents, are
\emph{divergent} if it has points $x<y$ such that $f(x)<x$ and $f(y)
>y$; a cycle (pattern) that is not divergent is
\emph{convergent}. Clearly, a convergent cycle has a unique basic
interval $U$ such that its left endpoint is mapped to the right and its
right endpoint is mapped to the left. Evidently, this interval contains
an $f$-fixed point, always denoted by $a$. There is an equivalent way
to define convergent patterns. Namely, if $f$ is a $P$-monotone map for
a cycle $P$ then $P$ is convergent if and only if $f \in \mathcal{U}$.

Let $P$ be a cycle of $f \in \mathcal{U} $ and $\varphi=\varphi_a$ be a
function defined above. Following \cite{bk98, BM2}, define the
\emph{code} for $P$ as a function $L:P \to \mathbb{R}$ given by
$L(x)=0$ for the leftmost point $x$ of $P$ and then by induction
$L(f(y))= L(y) + \rho - \varphi(y)$ where $\rho$ is the over-rotation
number of $P$. When we get back to $x$ along $P$, we $n$ times add the
number $\rho$ (here $n$ is the period of $P$), and we subtract the sum
of $\varphi$ along $P$ which by definition equals $n \rho$, so in the
end we do not change $L$ at $x$. Thus, the concept of code is
well-defined. If a cycle $P$ is convergent, the code for $P$ is
\emph{monotone} if for any $x,y \in P$ , $x>_a y$ implies $L(x)<L(y)$;
if we only require that $L(x)\le L(y)$, we say that the code of $P$ is \emph{non-decreasing}. Clearly, all these notions can be defined for patterns too.

\begin{theorem}[\cite{BM2}]\label{BM4}
A pattern is over-twist if and only if it is convergent and has
monotone code.	
\end{theorem}

We need results of \cite{BS}. Denote by $\mathcal{U}_1$ the family of
continuous maps $f:[0,1]\to [0,1]$ with a unique critical point $c$
which is a local maximum of $f$ and a unique fixed point $a$; call maps
from $\mathcal{U}_1$ \emph{unimodal}. For each $x\in [0,1]$, define its
\emph{itinerary} as the sequence $ i(x)= i_0(x) i_1(x) i_2(x) \dots$ of
symbols $L, C, R$ as follows:

 \begin{equation}
 i_j(x)=
 \begin{cases}

 C & \text{if $ f^j(x)= c$}\\
 L & \text{if $f^j(x) <  c$}\\
 R & \text{if $f^j(x) >  c$}\\

 \end{cases}
 \end{equation}

Order symbols $L, C, R$ as follows: $L < C < R $. Now order itineraries
as follows. Suppose $A = (a_0, a_1, a_2, \dots)$ and $B = (b_0, b_1,
b_2, \dots)$ are two itineraries such that $ j = $ min $\{ k\ge 0| a_k \neq b_k \} $. We write $A \succ B$ if there is an
even number of $R's $ among $a_0, a_1, \dots, a_{j-1}$ and $a_j> b_j $,
or there is an odd number of $R's$ among $a_0, a_1, \dots, a_{j-1}$ and
$a_j< b_j$ . If $x>y$, then $i(x) \succ i(y)$. Conversely, if $i(x)
\succ i(y)$, then $x> y$.

An itinerary $A$ is \emph{shift maximal} if $A \succ \sigma^j(A)
\forall j \in \mathbb{N} $ where $\sigma$ is the left shift. For a
continuous map $f\in \mathcal{U}_1$ define the \emph{kneading sequence}
$\nu(f)$ of $f$ as the itinerary of $f(c)$, that is, the sequence
$i(f(c)) $. Clearly, $\nu(f)$ is shift maximal. The over-rotation
interval $[r_f, 1/2]$ of a given unimodal map $f$ depends only on its
kneading sequence $\nu(f)$ so that we can write $r_{\nu(f)}$ instead of
$r_f$. In fact, if $f$ and $g$ are two unimodal maps with $\nu(f)\succ
\nu(g)$, then $I_f\supseteq I_g$.

Let us now describe the unimodal over-twist pattern
$\gamma_{\frac{p}{q}}$ of over-rotation number $\rho=p/q$ where $p$ and
$q$ are coprime:

\begin{enumerate}

\item the first $q-2p$ points of the orbit from the left are
    shifted to the right by $p$ points;

\item the next $p$ points are flipped (that is, the orientation is
    reversed, but the points which are adjacent remains adjacent)
    all the way to the right;

\item the last $p$ points of the orbit on the right are flipped all
    the way to the left.

\end{enumerate}

Thus $\gamma_{\frac{p}{q}}$ is described by the following permutation
$\pi_{\frac{p}{q}}$:

\begin{equation}
\pi_{\frac{p}{q}}(j)=
\begin{cases}

j+p  &   \text{if $1 \leq j \leq q-2p $}\\
2q-2p+1-j  & \text{if $q-2p+1 \leq j \leq q-p$}\\
q+1-j & \text{if $q-p+1\leq j \leq q$ }\\

\end{cases}
\end{equation} 		

\noindent The unimodal over-twist pattern $\gamma_{\frac{2}{5}}$ is
shown in the next figure.

\begin{figure}[H]
 	\caption{\textbf{\textit{The Unimodal over-twist pattern $\gamma_{\frac{2}{5}}$ }}}
 	\centering
 	\includegraphics[width=0.7\textwidth]{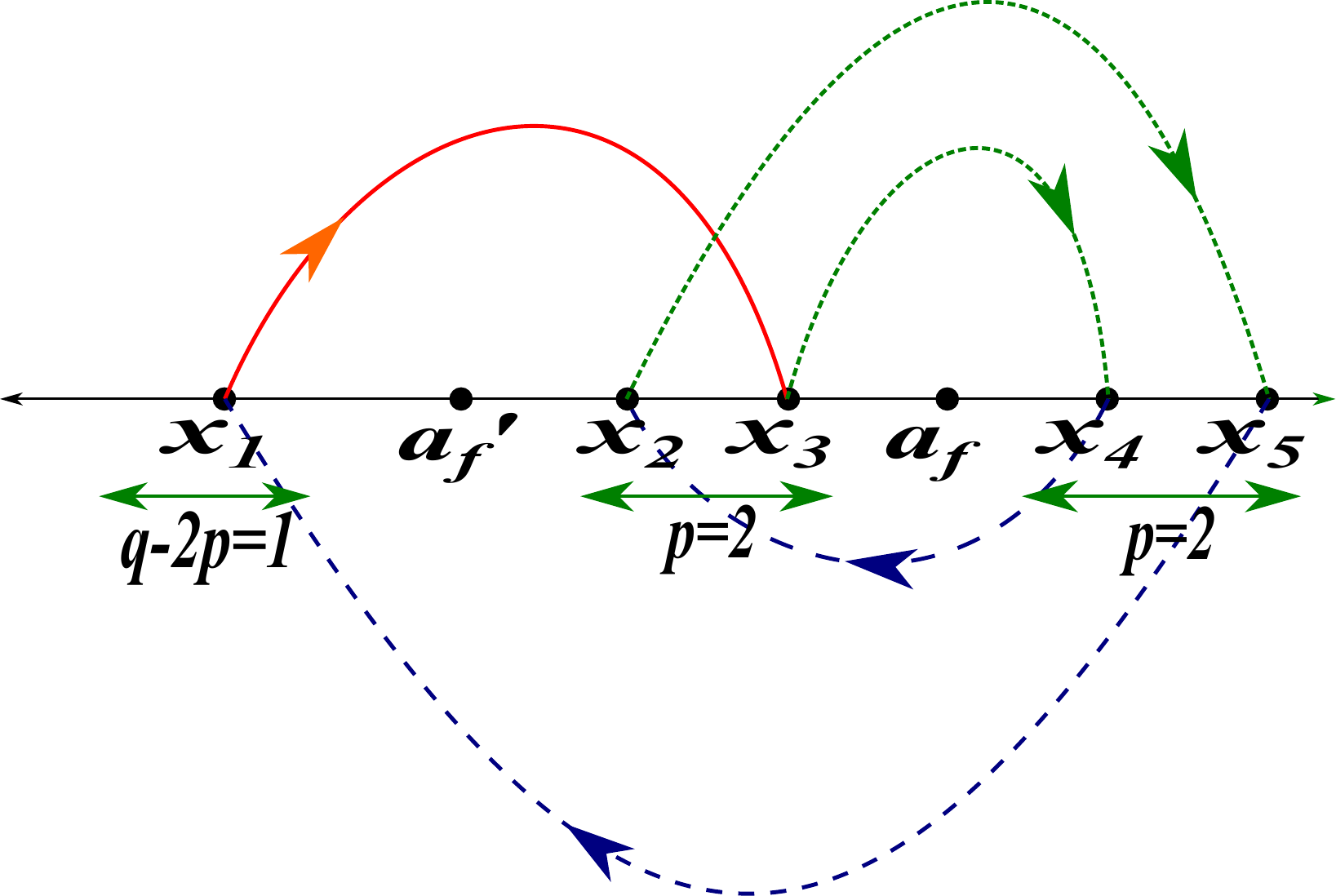}
\end{figure}

\begin{theorem}[\cite{BS}]\label{weak:knead}
The only unimodal over-twist pattern of over-rotation number $\rho$ is
the pattern $\gamma_{\rho}$. Thus, if $f:[0,1] \to [0,1]$ is a unimodal
map and the over-rotation interval of $f$ is $I_f = [r_f, \frac{1}{2}]$
then $r_f\le \rho$ iff $f$ has a periodic orbit with pattern
$\gamma_{\rho}$  	
\end{theorem}	
 	
We need more concepts from {\it one-dimensional combinatorial dynamics}
\cite{alm00}. A map $f$ has a {\it horseshoe} if there are two closed
intervals $I, J$ with disjoint interiors whose images cover their
union. In particular, $f$ has a horseshoe if there exist points $a,b,c$
such that either $f(c)\le a=f(a)<b<c\le f(b)$ (set $I=[a, b], J=[b,
c]$) or $f(c)\ge a=f(a)>b>c\ge f(b)$ (set $I=[b, a], J=[c, b]$). It is
easy to see \cite{BM1} that if a map has a horseshoe then it has
periodic points of all possible over-rotation numbers. %A {\it (cyclic)
%pattern} is the family of all cycles on the real line that induce the
%same cyclic permutation of the set $T_n=\{1,2,\dots,n\}$ or its flip;

A map (not necessarily one-to-one) of the set $T_n=\{1,2,\dots,n\}$
into itself is called a {\it non-cyclic pattern}. If an interval map
$f$ has a cycle $P$ from a pattern $\pi$ associated with permutation
$\Pi$, we say that $P$ is a {\it representative} of $\pi$ (or $\Pi$)
in $f$ and $f$ {\it exhibits} $\pi$ (or $\Pi$) on $P$; if $f$ is
{\it monotone (linear)} on
each complementary to $P$ interval, we say that $f$ is {\it
$P$-monotone ($P$-linear)} \cite{MN}. In what follows the same
terminology will apply to permutations, patterns and cycles, so for
brevity we will be introducing new concepts for, say, permutations.
%Observe also, that permutations are understood up to orientation.
%Finally, notice that in what follows we will interchangeably talk about
%permutations and patterns.

%A permutation $\pi$ is said to have a {\it block structure} if there is
%a collection of pairwise disjoint segments $I_0, \dots, I_k$ with
%$\pi(T_n\cap I_j)=T_n\cap I_{j+1}, \pi(T_n\cap I_k)=T_n\cap I_0$; the
%intersections of $T_n$ with intervals $I_j$ are called {\it blocks} of
%$\pi$.
%A permutation without a block structure is said to be with \emph{no
%block structure}, or, equivalently, {\it irreducible}. If we collapse
%blocks to points, we get a new permutation $\pi'$, and then $\pi$ is
%said to have a block structure {\it over $\pi'$}. A permutation $\pi$
%{\it forces} a permutation $\theta$ if any continuous interval map $f$
%which exhibits $\pi$ also exhibits $\theta$.

\section{Some results regarding the code}

Let us prove a general result relating the code and the property of a
pattern $\pi$ to force another pattern of smaller over-rotation number.
In what follows we use notation, say, $[b, d]$ for a closed interval
with endpoints $b$ and $d$ \emph{without assuming that $b<d$}.

\begin{lemma}\label{code:nonmono}
Let $f: [0,1] \to [0,1] $ be a continuous map with a unique fixed point
$a$. Let $P$ be a periodic orbit of $f$ with over-rotation pair
$(p,q)$. %where $p$ and $q$ may not be coprime.
If there exist $x, y\in P$ satisfying $x>_a y $ but $L(x)> L(y)$, then
%the over-rotation interval of $f$ is
$I_f \supset [\frac{l}{k}, \frac{1}{2}]$ where $\frac{l}{k}<
\frac{p}{q}$ and $k<q$ .
\end{lemma}

\begin{proof}
Choose $k \in \mathbb{N}, k < q$ with $ x = f^k(y)$. Consider the
sequence of intervals $ J_i , i \in \{0,1,2, \dots k \}$ such that $J_0
= [y,a], J_1 = [f(y),a], J_2=[f^2(y), a], \dots J_{k-1}=[f^{k-1}(y),
a], J_k=[f^k(y), a]=[x,a]\supset [y,a]  $ since $x
>_a y $. Clearly, $J_0 \supseteq J_1 \supseteq J_2 , \dots , \supseteq
J_{k-1} \supseteq J_k \supseteq J_0 $. Thus, $J_0 \to J_1 \to J_2 \to
J_3 \to \dots \to J_{k-1} \to J_0 $ is an admissible loop of intervals.
By Lemma \ref{BM3} there exists a periodic point $z \in J_0$ such that
$f^i(z) \in J_i, i\in \{0,1,2, \dots k-1\}$ and $f^k(z) = z$. The
period of $z$ is either $k$ itself, or a divisor of $k$; the
over-rotation number of $z$ is then $\frac{l}{k}$ for some integer $l$.
The code of the point $x$ is $L(x) = L(f^k(y))= L(y) + k\frac{p}{q} -
l$. Now, from $L(x) > L(y)$ we have $L(y) + k\frac{p}{q} - l > L(y)$
which implies that $\frac{p}{q}>\frac{l}{k}$, and we are done.
\end{proof}	

By Theorem \ref{BM4} a periodic orbit $P$ from Lemma \ref{code:nonmono}
is not an over-twist. However Lemma \ref{code:nonmono} shows more; it
shows that $P$ forces a periodic orbit of smaller rotation number and
smaller period.

Let us now study the case when the code is only non-decreasing (and not
strictly increasing as in Lemma \ref{code:nonmono}).

%We now show that if in a periodic orbit $P$, if $\forall \, x,y \in P x
%>_a y \implies L(x) < L(y)$ and there exist $u, v \in P $, $u\neq v$
%such that $ u >_a v $ and $L(u) = L(v)$, then the over-rotation pair of
%$P$ is not coprime.

\begin{lemma}\label{noncoprime:ovp}
Let $f: [0,1] \to [0,1] $ be a continuous map with a unique fixed point
$a$. Let $P$ be a periodic orbit of $f$ of over-rotation pair $(p, q)$.
Suppose that % for any $x, y \in P$, $x >_a y$ implies $L(x) \leq L(y)$
there exist $u, v\in P $, $u\neq v$ such that $u >_a v $ and $L(u) =
L(v)$. Then $p$ and $q$ are not coprime.
\end{lemma}

\begin{proof}
Clearly, $u=f^s(v)$ for some $m \in \mathbb{N}, s < q$. Since $ u
>_a v $, then $[v,a] \to [f(v),a] \to [f^2(v),a] \to [f^3(v), a] \to
\dots \to [f^{s-1}(v), a] \to [v,a ]$ is an admissible loop of
intervals. Thus, by Lemma \ref{BM3} there exists a periodic point $x
\in [v,a ]$ such that $f^i(x) \in [f^i(v),a], i \in \{0,1,2, \dots s-1
\} $ and $f^s(x) = x$. The period of $x$ is either $s$ itself, or a
divisor of $s$; the over-rotation number of $x$ is then $\frac{r}{s}$
for some integer $r$. Then, the code of the point $u$ is $L(u) =
L(f^s(v))= L(v) + s\frac{p}{q} - r $. Since, $L(u) = L(v)$, we have
$L(v) + s\frac{p}{q} - r = L(v)$. This means $\frac{p}{q}=
\frac{r}{s}$. But, as we know, $s< q$. Thus, $p$ and $q$ are not
coprime.
\end{proof}

We are ready to estimate on how much the number $r_\pi$ differs from
the over-rotation number $\rho(\pi)$ of a pattern $\pi$ in the case
when the over-rotation pair of $\pi$ is coprime but $\pi$ is not
over-twist. Given a rational number $\frac{p}{q}$ let
$LF(\frac{p}{q})=\frac{l}{k}$ be the closest from the left rational
number to $\frac{p}{q}$ with denominator $k<q$.

\begin{cor}\label{c:farey} Suppose that $\pi$ is a convergent pattern with a
coprime over-rotation pair $(p, q)$; moreover, assume that $\pi$ is not
over-twist. Then $r_\pi\le LF(\frac{p}{q})$.
\end{cor}

\begin{proof}
Let $P$ be a cycle of pattern $\pi$, and let $a$ be a unique fixed
point of the $P$-linear map $f$. Since $\pi$ is not over-twist, the
code of $P$ is not monotone. By Lemma \ref{noncoprime:ovp} codes of
points of $P$ located on the same side of $a$ cannot coincide. By
Theorem \ref{BM4} the fact that $\pi$ is not over-twist implies that
the code on $P$ is not increasing. Hence the code strictly decreases on
two points of $P$ located on the same side of $a$, and then Lemma
\ref{code:nonmono} implies the desired.
\end{proof}

Observe that Corollary \ref{c:farey} extends the results of
\cite{alm98} (obtained for degree one circle maps and classical
rotation numbers of their cycle) to interval maps and over-rotation
numbers of their cycles with \emph{coprime over-rotation pairs}. As we
will see later, similar extension of the results of \cite{alm98} to
cycles whose over-rotation pairs are not coprime fails.

Now, we prove the main result of this section of the paper. %We prove a
%necessary condition for a periodic orbit to be over-twist.
	
\begin{theorem}\label{t:endpt-ot}
Let $P$ be a cycle of a convergent pattern $\pi$ such that the $P$-linear
map $f$ has the over-rotation interval $I_f = [\rho(P), \frac{1}{2}]$
(in other words, $r_P=\rho(P)$).
%where $\rho(P)$ is the over-rotation number of $P$.
Moreover, suppose that the over-rotation pair of $P$ is coprime. Then
the pattern $\pi$ is over-twist.
\end{theorem}	

\begin{proof}
%Let $P$ be a cycle of convergent pattern $\pi$ such that the $P$-linear
%map $f$ has the over-rotation interval $[\rho(P), \frac{1}{2}]$ where
%$\rho(P)$ is the over-rotation number of $P$. Moreover, suppose that
%the over-rotation pair of $P$ is coprime.
%We need to show that $\pi$ is over-twist.
Since $P$ has coprime
over-rotation pair, then by Lemma \ref{noncoprime:ovp} %, the codes of
%$P$ is either strictly increasing or strictly non-decreasing,that is,
there exist no  $u, v \in P$, $u\neq v $ with $ u >_a v$ but $L(u) =
L(v)$. If there exist $x, y \in P$ with $x >_a y$ and $L(x)> L(y)$,
then by Lemma \ref{code:nonmono} we have $I_f \supset [\frac{m}{n},
\frac{1}{2}]$ with $\frac{m}{n} < \rho(P)$. Thus, the code for $P$ is
strictly monotone, and by Theorem \ref{BM4} $\pi$ is over-twist.
\end{proof}	

In the next section we show that without the assumption that the
over-rotation pair of $P$ is coprime Theorem \ref{t:endpt-ot} is not
true.

\section{Very badly ordered periodic orbits of interval maps}

By Theorem \ref{t:endpt-ot} for coprime patterns $\pi$ the fact that $r_\pi=\rho(\pi)$ implies that $\pi$ is an over-twist pattern. Hence the next corollary follows.

\begin{cor}\label{c:non-cop}
If a pattern $\pi$ is not an over-twist and $r_\pi=\rho(\pi)$ then the
over-rotation pair of $\pi$ is non-coprime.
\end{cor}

A natural conjecture would be then that for non-coprime patterns $\theta$ the fact that $r_\theta=\rho(\theta)$ implies that $\theta$ has a block structure of an over-twist pattern (of the appropriate over-rotation number). This would be analogous to the similar results for degree one circle maps (see \cite{alm98}). However we will show later that on the interval there exist patterns $\theta$ with $r_\theta=\rho(\theta)$ which do not have a block structure over over-twist patterns (by Theorem \ref{t:endpt-ot} such patterns $\theta$ \emph{must} have non-coprime over-rotation patterns). Observe, that,
as was remarked before, patterns $\pi$ that have a block structure over
over-twist patterns, have the desired property $r_\pi=\rho(\pi)$ (by
Theorem \cite{BM1}, over-twist patterns must be coprime, hence the
non-coprime patterns in question are not over-twists).

\begin{definition}\label{d:badly} A cyclic pattern is said to be \emph{very badly ordered} if its over-rotation number equals the
left endpoint of the forced over-rotation interval while the pattern has no block structure over an over-twist of the same over-rotation number.
\end{definition}

Our terminology is analogous to that from \cite{alm98}.

%Before we proceed further , let us prove a Lemma that will help us to
%describe \emph{badly ordered} patterns.

\begin{lemma}\label{code:nondec}
Let $P$ be a periodic orbit of over-rotation pair $(nr, ns)$ where $n,
r, s \in \mathbb{N} $,  $g.c.d(r,s) = 1$ and $ n > 1 $. Suppose the
code for $P$ is non-decreasing. Then $P$ does not force a periodic
orbit $Q$ of over-rotation number less than $\frac{r}{s}$.
\end{lemma}

\begin{proof}
Let $f$ be a $P$-linear map. By way of contradiction, suppose that
there exists a periodic point $x$ of $f$ with orbit $Q$ of
over-rotation pair $(m, n)$ such that $\frac{m}{n} < \frac{r}{s}$.
Consider \emph{all} periodic points $z\ne a$ such that for every $0\le
i\le n-1$ we have $f^i(x)\ge_a f^i(z)$. Since
$\frac{m}{n}<\frac{1}{2}$, points of such periodic orbits cannot come
too close to $a$. Clearly, this is a closed family of periodic orbits.
Choose the closest to $a$ among all of them point $z$. For each
$f^i(z)$ define $y_i$ as the closest to $a$ point on the same side of
$a$ as $f^i(z)$ such that $f(y_i)=f(f^i(z))$ coincides with $f^i(z)$.

We claim that $y_i=f^i(z)$ for every $i$. Indeed, otherwise we can
construct an admissible loop $[y_0, a]\to \dots [y_{n-1}, a]\to [y_0,
a]$ which, by Lemma \ref{BM3}, will give rise to a periodic point
$z'\in [y_0, a]\subset [z, a]$ such that $f^i(z)\ge_a f^i(z')$ for each
$i$; by the choice of $z$ we get $z'=z$ which immediately implies the
claim. Thus, for every $i$ the only point that belongs to $[f^i(z), a]$
and has the same image as $f^i(z)$ is the point $f^i(z)$ itself. We can
assume that the originally chosen point $x$ equals $z$.

If $t\in Q$ then $t$ belongs to a unique $\{P\cup \{a \}\}$-basic
interval, say, $[u, v]$; assume that $u>_a t>_a v$. We claim that
$f(v)>_a f(t)$ is impossible. Indeed, otherwise there exists a point
$y\in (t, a)$ with $f(y)=f(t)$, a contradiction with the choice of $Q$.
This implies that $f(u)>_a f(t)$. Indeed, suppose otherwise. Then both
$f(u)\not >_a f(t)$ and $f(v)\not >_a f(t)$ which means that
$f(t)\notin [f(u), f(v)]$, a contradiction with the fact that $f$ is
$P$-linear. Thus, $f(u)>_a f(t)$. Set $u=\xi(t)$ and do this for every
$t\in Q$. Then $[\xi(x),a]\to [\xi(f(x)), a]\to [\xi(f^2(x)), a] \to
\dots \to [x,a]$ is an admissible loop which we will call $\beta$;
observe that all points  $\xi(f^i(x))$ belong to $P$ and that the
over-rotation number of this admissible loop is $\frac{m}{n}$. Let
$v_0=\xi(x)$, $v_1=\xi(f(x)), v_2=\xi(f^2(x))$ and so on. Thus, $\beta$
: $[v_0,a] \to [v_1, a] \to [v_2, a ] \to \dots [v_0,a]$ where $v_i\in
P$ and $f(v_i)>_a v_{i+1},$ $0\le i\le n-1$.

If $v_{i+1} \neq f(v_i)$ for some $i$, then $v_{i+1}=f^j(v_i)$ for some
$j > 1$. In this case replace the arrow $ [v_i,a] \to [v_{i+1}, a ]$
with the block $[v_i,a] \to [f(v_i), a] \to [f^2(v_i), a] \to
[f^3(v_i), a] \to \dots \to  [f^{j-1}(v_i), a] \to [v_{i+1}, a ]$. Do
this for every $i$ with $v_{i+1} \neq v_i$ . Then, we get a new
admissible loop $ \gamma$ which is either the fundamental loop of $P$
or its repetition and so the over-rotation number of $\gamma$ must be
equal to that of $P$, that is $\frac{r}{s}$ .

Let us now compare the over-rotation numbers of the admissible loops
$\beta $ and $\gamma$. Define a function $\phi: P \to \{0,1 \}$ so that
$\phi(x)=1$ if $x>a$ and is mapped to the left of $a$ and $\phi(x)=0 $
otherwise. The over-rotation number of $\gamma$ is the weighted average
of the over-rotation number of $\beta$ and inserts. Here by the
over-rotation number of an  insert:  $[v_i,a] \to [f(v_i), a] \to
[f^2(v_i), a] \to [f^3(v_i), a] \to \dots \to [f^{j-1}(v_i), a] \to
[v_{i+1}, a]$ we mean the average of the values of the function $\phi$
at the points $f(v_i), f^2(v_i), \dots f^{j-1}(v_i)$.  Take an insert
$[v_i,a] \to [f(v_i), a] \to [f^2(v_i), a] \to [f^3(v_i), a] \to \dots
\to  [f^{j-1}(v_i), a] \to [v_{i+1}, a ]$  and let its over-rotation
number be $t$. Since $f(v_i) >_a v_{i+1}$ and the code for $P$ is
non-decreasing, then $L(v_{i+1}) - L(v_i) \ge 0$. Hence, $t\le
\frac{r}{s}$. Thus, the over-rotation number of every insert is at most
$\frac{r}{s}$.

The over-rotation number $\frac{r}{s}$ of $\gamma$ is the weighted
average of the over-rotation number $\frac{m}{n}<\frac{r}{s}$ of
$\beta$ and the over-rotation number of inserts each of which is at
most $\frac{r}{s}$, a contradiction that implies the claim.
\end{proof}

\begin{figure}[H]
	\caption{\textbf{\textit{A unimodal very badly ordered periodic orbit of over-rotation pair  $(2,6)$ }}}
	\centering
	\includegraphics[width=1\textwidth]{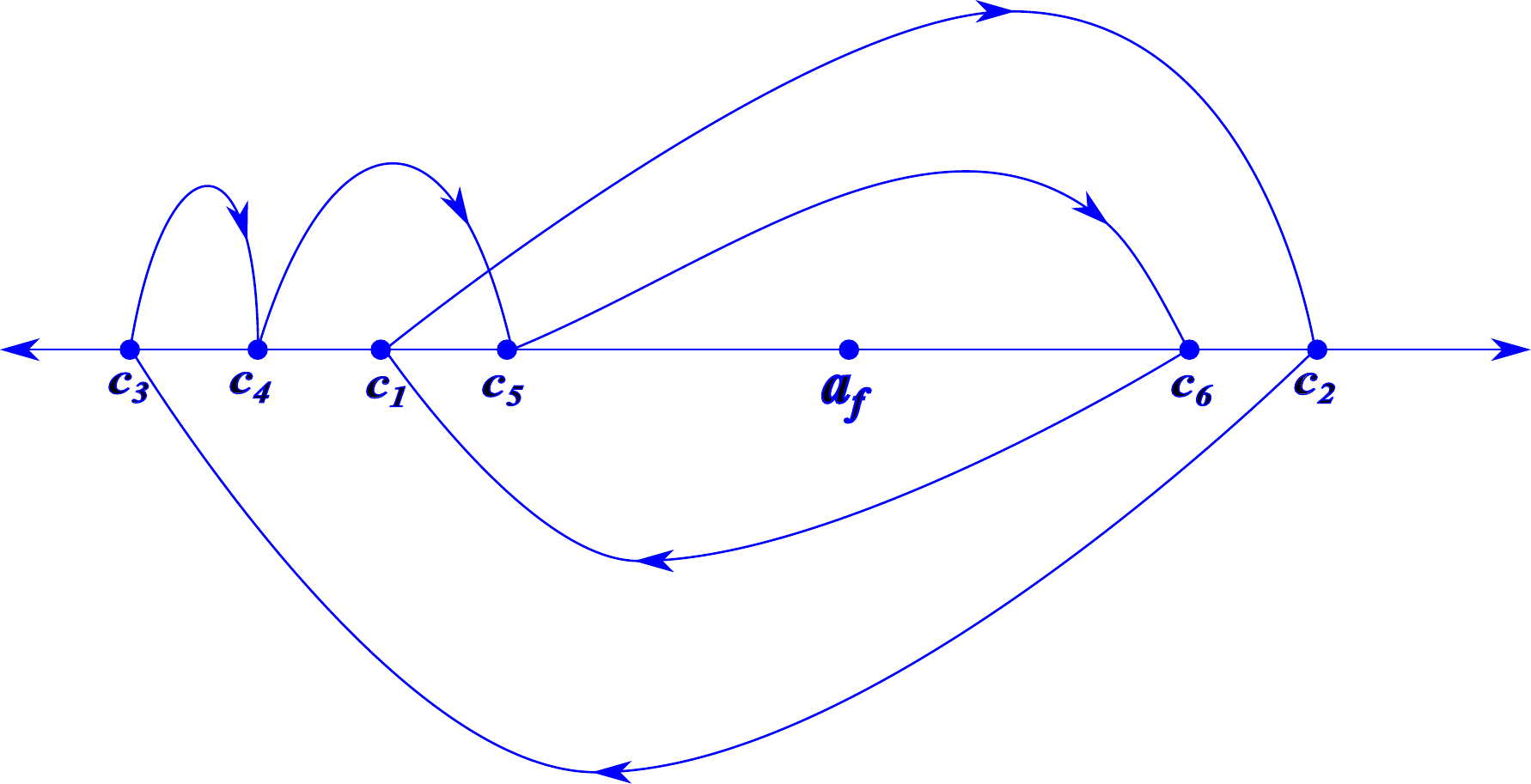}
\label{f:badly}
\end{figure}

Lemma \ref{code:nondec} and Lemma \ref{code:nonmono} immediately imply Corrollary \ref{c:alt-badly}.

\begin{cor}\label{c:alt-badly}
Let $\pi$ be a convergent pattern. Then $r_\pi=\rho(\pi)$ if and only if the code of $\pi$ is non-decreasing.
\end{cor}

Example \ref{e:badly} describes a very badly ordered pattern of over-rotation pair $(2, 6)$.

\begin{example}[see Figure \ref{f:badly}]\label{e:badly}
Consider the pattern $\pi $ defined by the permutation
$\pi=(1, 2, 4, 5, 3, 6, 1)$. It is unimodal pattern with over-rotation pair
$(2, 6)$. If we compute out the codes of the points of the periodic
orbit, we will see that the code is monotonically non-decreasing. But,
$\pi$ does not have a block structure over an over-twist periodic orbit.
\end{example}

Suppose that $f:\uc\to \uc$ has a cycle $P$ of classical rotation pair
$(mp, mq)$ which does not have a block structure over a rotation by
$p/q$. Then by \cite{alm98} the rotation interval of $f$ contains $p/q$
in its interior. 	 This shows the difference in properties of certain
periodic orbits on the interval and on the circle.

\section{Strongest unimodal pattern corresponding to a given over-rotation interval}

Our purpose now is to show that very badly ordered interval patterns
$\pi$ do exist. By Corollary \ref{c:non-cop} if $\pi$ is a very badly
ordered pattern then $orp(\pi)=(np, nq)$ with $p, q$ coprime and $n>1$.
We will construct a unimodal very badly ordered pattern for any numbers
$n, p, q$ satisfying the above conditions as well as inequality $2p<q$
(one can show that if $2p=q$ then the pattern $\pi$ must have a block
structure of a periodic orbit of period $2$). In this section we do
some preparatory work by studying unimodal maps. To this end we use
kneading sequences.

\subsection{The strongest kneading sequence associated with a given over-rotation interval}

To describe the kneading sequence associated with $\gamma_\rho$ (and
following from its description given right before Theorem
\ref{weak:knead}) we need to introduce new notation. Consider the shift
$\zeta_{\rho}$ by $\rho $ on $[0,1]$ modulo 1 assuming $\rho <
\frac{1}{2}$.  Define the kneading sequence $\nu_{\rho}= \{
\nu_{\rho}(0), \nu_{\rho}(1), \dots \} $ as follows. For each $n\in
\mathbb{N} $, define:

 \begin{equation}\label{e:knead-twist}
 \nu_{ \rho }(n)=
 \begin{cases}

 C & \text{if $ \zeta_\rho^{n}(\rho)=0 $}\\
 R & \text{if $ 0<\zeta_\rho^{n}(\rho)<2\rho $}\\
 L & \text{if $\zeta_\rho^{n}(\rho) \geq 2\rho$}\\

 \end{cases}
 \end{equation}

It is easy to see that $\nu_\rho$ is the kneading sequence associated
with $\gamma_\rho$ (we leave to the reader a straightforward
verification of this fact).

By Theorem \ref{weak:knead} $\nu_\rho$ is the least (the weakest)
kneading sequence of a unimodal map $f$ such that $\rho \in I_{f}$.
Evidently, if $s, t \in \mathbb{Q}$ and $s<t$, then $\nu_{s} \succ
\nu_{t}$. Now, we will find the strongest kneading sequence of a
unimodal map $f$ such that $I_f=[\rho, \frac{1}{2}]$.

Consider the case when $\rho $ is rational. Let $\rho = \frac{p}{q}$.
Take a rational number $s<\frac{p}{q}$ %$\widetilde{\frac{p}{q}}<\frac{p}{q}$
very close to $\frac{p}{q}$. Then $\nu_s\succ \nu_{\frac{p}{q}}$.
%$\nu_{\widetilde{\frac{p}{q}}}\succ \nu_{\frac{p}{q}}$.
%Now, take $s$ %$\widetilde{\frac{p}{q}}$
%closer and closer to $\frac{p}{q}$ and look at what happens with the
%kneading sequences.
As $s$ %$\widetilde{\frac{p}{q}}$
is getting closer and closer to $\frac{p}{q}$, we get better and better
idea about the strongest kneading sequence corresponding to the over-rotation interval $[\frac{p}{q},\frac{1}{2}]$. The greatest (strongest)
kneading sequence $\nu_{\rho}'= (\nu_{\rho}'(0), \nu_{\rho}'(1), \nu_{\rho
}'(2),\dots)$ of a unimodal map $f$ with $I_f=[\rho,
\frac{1}{2}]$ is then given by the limit of $\nu_s$
%$\nu_{\widetilde{\frac{p}{q}}} $ as $s$ %$\widetilde{\frac{p}{q}}$
as $s$ approaches $\frac{p}{q}$ from the left.
%Note that, $\nu_{\rho}(0)= C $,  $\nu_{\rho}(1)= R$, $\nu_{\rho}(2)= C
%$, $\dots$, $\nu_{\rho}(q)= C $, $\nu_{\rho}(q+1)= R $,
%$\nu_{\rho}(q+2)= L$ and so on.
Observe that the equation \ref{e:knead-twist} applies to the kneading sequences $\nu_s$, and use this fact in order to give an explicit description of $\nu_{\rho }'$ (recall that $s<\rho, s\to \rho$).

%Then, when comparing
%$\nu_{\rho}$ and $\nu_s$ %{\widetilde{\rho}}$
%and taking the limit of $\nu_s$ %$\nu_{\widetilde{\frac{p}{q}}} $
%as $s$
%$\widetilde{\frac{p}{q}}$
%approaches $\frac{p}{q}$ from the left we obtain the following:

Indeed, by continuity it is easy to see that $\nu_s(i)=\nu_\rho(i),$ $0\le i\le q-2$; in particular, $\nu_s(0)=\nu_\rho(0)=R$ while $\nu_s(1)=\nu_\rho(1)=L$. On the other hand, it follows that, while $\zeta_\rho^{q-1}(\rho)=0$ and, therefore, $\nu_\rho(q-1)=C$, the point $\zeta_s^{q-1}(s)$ belongs to $[2s, 1)$ (and is located slightly to the left of $1$); this implies that $\nu_s(q-1)=L$. Then $\zeta^q_\rho(\rho)=\rho$ and, hence, $\nu_\rho(q)=R$; again, by continuity, $\zeta_s^q(s)$ is located slightly to the left of $\zeta^q_\rho(\rho)=\rho<2s$ and,  hence, $\nu_s(q)=R$. On the next step $\zeta_\rho^{q+1}=2\rho$ which, as we saw before, yields $\nu_\rho(q+1)=L$. However, $\zeta_s^{q+1}(s)$ is located slightly to the left of not only $2\rho$, but also $2s$ implying that $\nu_s(q+1)=R$. From this moment the dynamics of points and their mutual location will be repeat over and over (as $s\to \rho$, the number of repetitions will grow to infinity). This justifies the following explicit description of the kneading sequence $\nu'_\rho$:

\begin{enumerate}
	\item $\nu_{\rho}'(i)=\nu_{\rho}(i)$ for $i=0, 1, 2, \dots,
q-2$ (by continuity);
	\item $\nu_{\rho}'(q-1)=L$ while $\nu_\rho(q-1)=C$;
    \item $\nu_{\rho }'(q)=\nu_{\rho}(q)=R$;
    \item $\nu_\rho'(q+1)=R$ while $\nu_\rho(q+1)=L$;
    \item $\nu_\rho'(j)=\nu_\rho'(j-q)$.
\end{enumerate}

In short, the kneading sequence $\nu'_\rho$ is obtained from the kneading sequence $\nu_\rho$ as follows. First, we keep the initial two symbols $RL$.
From the next place on the kneading sequence $\nu'_\rho$ is a periodic adjustment of $\nu_\rho$ done as follows: each occurrence of the fragment $CRL$ in $\nu_\rho$ is replaced by $LRR$ in $\nu'_\rho$.

%\item $\nu_{\rho }'(q)=\nu_{\rho}(q)$
%	\item $\nu_{ \rho }'(q+1)=$
%	\item $\nu_{ \rho }'(q+3) =\nu_{ \rho }'(3)=\nu_{\rho}(3)$ and
%so on.
%\end{enumerate}

%Thus, $ \nu_{ \rho }'= (C,R,L,\nu_{\rho}(3),\nu_{\rho}(4) \dots
%\nu_{\rho}(q-1),L,RR,\dots)$ greatest(strongest) kneading sequence
%which gives the over-rotation interval precisely coinciding with
%$[\rho, \frac{1}{2}]$.

\subsection{The dynamics associated with $\nu'_\rho$}\label{ss:strong-dyn}

The dynamics of the corresponding unimodal map can be described as
follows. First consider a unimodal cycle $P$ which exhibits the unimodal over-twist pattern $\gamma_{\frac{p}{q}}$ of over-rotation number $\frac{p}{q}$ and denote the $P$-linear map by $f$. Let
$c$ be a unique critical point of $f$; for brevity let $c_i$ stand for $f^i(c)$. To construct,$\gamma_{\rho}'$ we add some points and define a new map on them so that together with $P$ they will form a new (non-periodic) pattern $\gamma_{\rho}'$ associated with the kneading sequence $\nu_\rho'$.

First, let us choose a point $c'$ slightly to the left of $c$ so that $c'$ belongs the $P$-basic interval with the right endpoint $c$. We define a map $g$ at $c'$ so that $g(c')=c_1'>f(c)$. Then we define $g$ at $c'_1$ so that $g(c_1')=c_2'<c_2$ (the points $c'_1$ and $c'_2$ correspond to the initial non-repetitive $RL$ of the kneading sequence $\nu'_\rho$). Note, that the map $f$ maps the point $c_2$ to the $p+1$-st point of $P$ counting from the left (we are now using spatial labeling). Next we define the map $g$ at $c_2'$ so that $g(c'_2)$ is the $p$-th point of $P$ (again, counting from the left and using spatial labeling). Otherwise, i.e. on the set $P$, we set $g=f$.
Clearly, this gives rise to a non-periodic unimodal pattern $\pi$ of the map $g$ restricted on $P\cup \{c', c_1', c_2'\}$.

Let us show that the kneading sequence of $\pi$ is $\nu_\rho'$. Observe that $c_1'$ is the point at which $g$ reaches its maximum. Hence when describing the itinerary of a point under $g$ we need to compare the location of a current point of $P$ and the point $c_1'$. Now, in the kneading sequence $\nu_\rho$ the symbol $\nu_\rho(2)$ is associated with the point $c_3$. By the description of the unimodal over-twist patterns the point $c_3$ is the $p+1$-st point (counting from the left) of $P$, and the kneading sequence $\nu_\rho$ from the symbol $\nu_\rho(2)$ is simply the itinerary of $c_3$. In the kneading sequence associated with $P$ the corresponding to place $\nu_\rho(2)$ place is occupied by the $p$-th point $z$ (counting from the left) of $P$. We need to compare the itinerary of $c_3$ under $g$ (coinciding with $f$ on $P$) with the itinerary of $z$ under $g$. Observe that $z<c_3$ are adjacent points of $P$.

By the properties of the unimodal over-twist patterns the points $c_3$ and $z$ for a while stay adjacent and are both located either to the left of $c_1'$ or to the right of $c_1'$ as we apply growing powers of $g$ to them. However at the moment when $c_3$ maps to $c$ this changes. To see how it changes, observe that on our pair of points $\{z, c_3\}$ the map $g$ acts so that they enter $[c', c_1']$ an even number of times. Hence at the time when $c_3$ maps to $c$ the corresponding iteration of $g$ is monotonically increasing on $\{z, c_3\}$. It follows that at this point the $g$-image of $z$ is to the left of $c'$. Thus, the corresponding symbol in the kneading sequence associated with $g$ must be $L$ as in $\nu'_\rho$. Then it easily follows from the dynamics of unimodal over-twist patterns that the next two symbols in the kneading sequence of $g$ will be $RR$, again as in $\nu'_\rho$. Because of periodicity it follows that $\nu'_\rho$ is, indeed, the kneading sequence associated with $g$ and the pattern $\pi$. Hence $\pi=\gamma'_\rho$ is the desired pattern of the strongest unimodal map $F$ with $I_F=[\frac{p}{q}, \frac12]$.

The construction of the kneading sequence $\nu'_\rho$ implies that for
the corresponding map $f$ we have $I_f\supset [\frac{p}{q},
\frac{1}{2}].$ On the other hand, suppose that $I_f\ne [\frac{p}{q},
\frac{1}{2}].$ Then $f$ must have a periodic orbit of over-rotation
number $s<\frac{p}{q}$ which by construction implies that its kneading
sequence must be stronger $\nu'_\rho$, a contradiction.

Together with the results of \cite{BS} this implies the next theorem.

\begin{theorem}
If $f:[0,1]\to [0,1]$ is unimodal, then $I_{f}=[\frac{p}{q},
\frac{1}{2}]$ if and only if $\nu(f)$ satisfies
$\nu_{\frac{p}{q}}'\succ \nu(f) \succ \nu_{\frac{p}{q}}$. 	
\end{theorem}

Figure \ref{f:strongest} illustrates the dynamics of the pattern
$\gamma'_{\frac{2}{5}}$.

\begin{figure}[H]
	\caption{\textbf{\textit{The strongest pattern corresponding to the over-rotation interval $[\frac{2}{5},\frac{1}{2}] : \gamma_{\frac{2}{5}}'$ }}}
	\centering
	\includegraphics[width=0.7\textwidth]{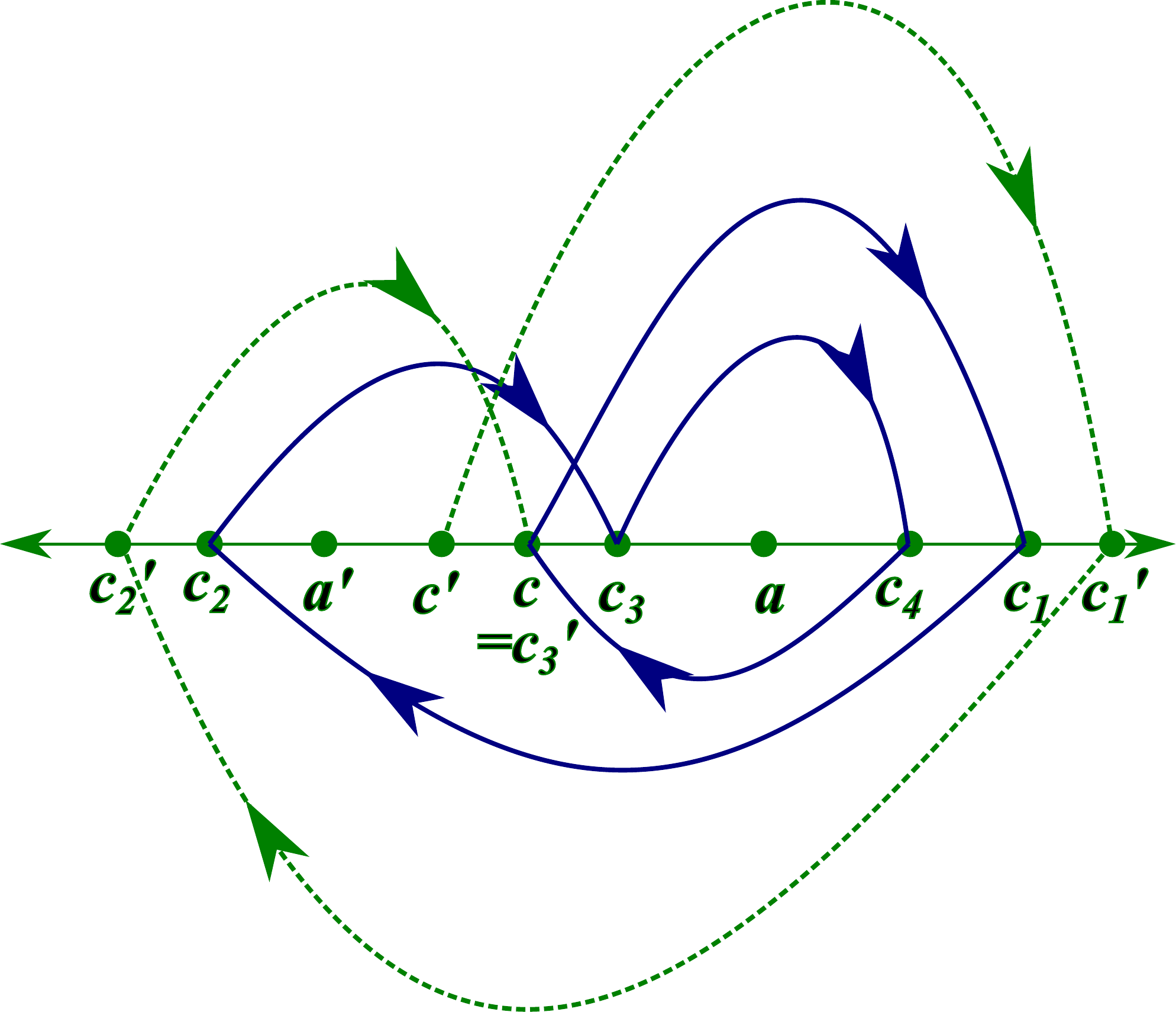}
\label{f:strongest}
\end{figure}

\section{Unimodal very badly ordered periodic orbits of arbitrary non-coprime over-rotation pair with over-rotation number less than $\frac{1}{2}$} % $(kp,kq)$ where $k,p,q \in \mathbb{N} $ and $ k > 1 $ }

It follows from the above that any very badly ordered cycle must have non-coprime over-rotation pair of over-rotation number less than $\frac{1}{2}$. In this section we show that very badly ordered cycles exist for \emph{any} non-coprime over-rotation pair with over-rotation number less than $\frac{1}{2}$. The description of the strongest unimodal over-twist pattern $ \gamma_{\rho}'$ corresponding to the over-rotation interval $[\rho, \frac{1}{2}]$ obtained in Section 4 is instrumental in the justification of our construction.

Suppose $k, p, q\in \mathbb{N}$, $g.c.d(p,q)=1$ and $k>1$. We want to construct a badly ordered periodic orbit of over-rotation pair $(kp,kq)$.
First we construct $k$ periodic orbits $P_1$, $P_2$, $\dots P_k$ with the following properties.

\begin{enumerate}
\item  $P_1=\{c_{11}, c_{12}, \dots, c_{1q}\}$, $P_2=\{c_{21}, c_{22}, \dots, c_{2q}\}$, $\dots,$  $P_k=\{c_{k1},$ $c_{k2},$ $\dots,$ $c_{kq}\}$ each of which exhibits pattern $\gamma_{\frac{p}{q}}$ and suppose the orbits have temporal labeling, that is, $ f(c_{ij}) = c_{i(j+1)}$ for $i \in \{1,2,\dots, k\}$ and $j\in \{1, 2, \dots, q\}$. Also, let critical points of the orbits $P_1$, $P_2,$ $\dots,$ $P_k$ be $c_{11}, c_{21}, c_{31}, \dots, c_{k1}$ respectively, and set $\mathcal P=P_1\cup \dots \cup P_k$.
\item

The orbits are placed  in such a manner that the $\mathcal P$-linear map which we call $f$ is unimodal and $c_{11}$ is the point of absolute maximum of $f$ and for every $i \in \{1, 2, \dots, q \}$, $c_{1i}>_a c_{2i}>_a \dots>_a c_{ki}$. Moreover,
iterations $f, f^2, \dots, f^q$ are monotone when restricted on
$[c_{11}, c_{k1}]$. The structure of over-twist patterns implies that
$f^q$ is increasing on $[c_{11}, c_{k1}]$, hence the construction is
consistent. We call $\mathcal P$ with the map $f$ defined as above the \emph{$k$-tuple lifting} of $\gamma_{\frac{p}{q}}$.
\end{enumerate}

The second step in our construction is to change the map $f$ to form a new map $g$ so that the $k$ cycles are glued into one periodic orbit under our new map $g$ while the map $g$ remains unimodal and the code of the resulting cycle is non-decreasing. Let $c_{1l},$ $c_{2l},$ $\dots,$ $c_{kl}$ be the $p$-th points (counting from the left) in the orbits $P_1$, $P_2$, $\dots P_k$ respectively (now we are using the spatial labelling). Let the new map $g$ acting on $\mathcal P$ be defined as follows.

First, we set $g(c_{13})= c_{2l}$ (the leftmost point of the outermost
orbit $P_1$ will be mapped to the $p$-th point, in the spatial
labelling, of $P_2$). This means that the point $c_{2(l-1)}$ that used
to be mapped to $c_{2l}$ under $f$ will have to be mapped elsewhere. We
set $g(c_{2(l-1)})=c_{3l}$, i.e. we move the $f$-image of $c_{2(l-1)}$
one step to the right. Thus, the point $c_{3(l-1)}$ that used to be
mapped to $c_{3l}$ under $f$ has to be mapped elsewhere. We set
$g(c_{3(l-1)})=c_{4l}$. In other words, each point $c_{j(l-1)}$ that
used to be mapped to $c_{jl}$ by $f$ has its image shifted to the right
by one and, thus, will be mapped by $g$ to the point $c_{(j+1)l}$ which
is one step to the right of its ``old image'' $c_{jl}$. This
construction applies to any $j=2, \dots, k$. In particular, on the last
step in the construction the point $c_{k(l-1)}$ that used to be mapped
to $c_{kl}$ by $f$ must have its image shifted to the right by one
which brings the $g$-image of this point to $c_{14}$, the point of
$P_1\cup \dots \cup P_k$ located immediately to the right of $c_{kl}$
and coinciding with the $f$-image of $c_{13}$. On other points of
$\mathcal P$ the maps $f$ and $g$ are the same. In this way we merge
the $f$-orbits $P_1, \dots, P_k$ into one $g$-orbit $\mathcal P$.

We claim that $g|_{\mathcal P}$ is the desired unimodal very badly ordered periodic orbit of over-rotation pair $(kp, kq)$. Indeed, it immediately follows from the construction that $g|_{\mathcal P}$ is unimodal and has the over-rotation pair $(kp, kq)$. Hence by Theorem \ref{BM3} $I_g\supset [\frac{p}{q}, \frac{1}{2}]$. To prove that the kneading sequence of $g$ is weaker than $\nu'_{\frac{p}{q}}$ we use the description of the dynamics associated with $\nu'_{\frac{p}{q}}$ given in Subsection \ref{ss:strong-dyn}.
Denote the kneading sequence of $g$ by $\nu_g$. Then the straightforward verification shows that the $\nu'_{\frac{p}{q}}$ and $\nu_g$ coincide until the symbol number $kq-1$ in both sequences. In $\nu_g$ this symbol is $C$ whereas in $\nu'_{\frac{p}{q}}$ this symbol is $L$. However one can easily check that before this moment there is an odd number of symbols $R$ in both sequences. Thus, $\nu'_{\frac{p}{q}}\succ \nu_g$ as desired.

Observe that a simple direct computation also shows that the code of $g$ is non-decreasing. Hence by Lemma \ref{code:nondec} one can also see that $I_g=[\frac{p}{q}, \frac{1}{2}]$.

\begin{example}
We now illustrate the above algorithm by constructing a very badly
ordered periodic orbit of over-rotation pair $(kp, kq)$ where $k=2$, $p=2$ and $q=5$, that is, of over-rotation pair $(4, 10)$. To this end, let us take two periodic orbits $P_1=\{c_1, c_2,$ $\dots, c_5\}$ and $P_2=\{d_1, d_2\dots, d_5\}$ each of which exhibits pattern $\gamma_{\frac{2}{5}}$. Give the orbits the temporal labeling, that is, $f(c_{i})=c_{i+1}$ and $f(d_{i})=d_{i+1}$ for $i\in \{1,2,\dots 5\}$. Let the critical points of the orbits $P_1$ and  $P_2$ be $c_1$ and $d_1$ respectively. Let $\mathcal{P}=P_1\cup P_2$ and $f$ be the $\mathcal{P}$-linear map. We place the periodic orbits $P_1$ and $P_2$ in such a manner that $f$ has a unique critical point $c_1$ which is the absolute maximum of $f$ and for every $i \in \{1, 2, \dots , 5\}$, $c_i>_a d_i$. Moreover, iterations $f, f^2, \dots, f^5$ are monotone when restricted to $[c_1,d_1]$ and $f^5$ is increasing on $[c_1, d_1]$. The set $\mathcal{P}$ with the map $f$ as defined above will be called the \emph{doubleton lifting} of the cycle $\gamma_{\frac{2}{5}}$; it is shown in Figure \ref{fig:double} with $P_1$ represented by the %red
dotted line and $P_2$ by the %blue
smooth line.

\begin{figure}[H]
	\caption{\textbf{\textit{Construction of doubleton lifting of the over-twist periodic orbit of over-rotation number $\frac{2}{5}$ }}}
	\centering
	\includegraphics[width=1\textwidth]{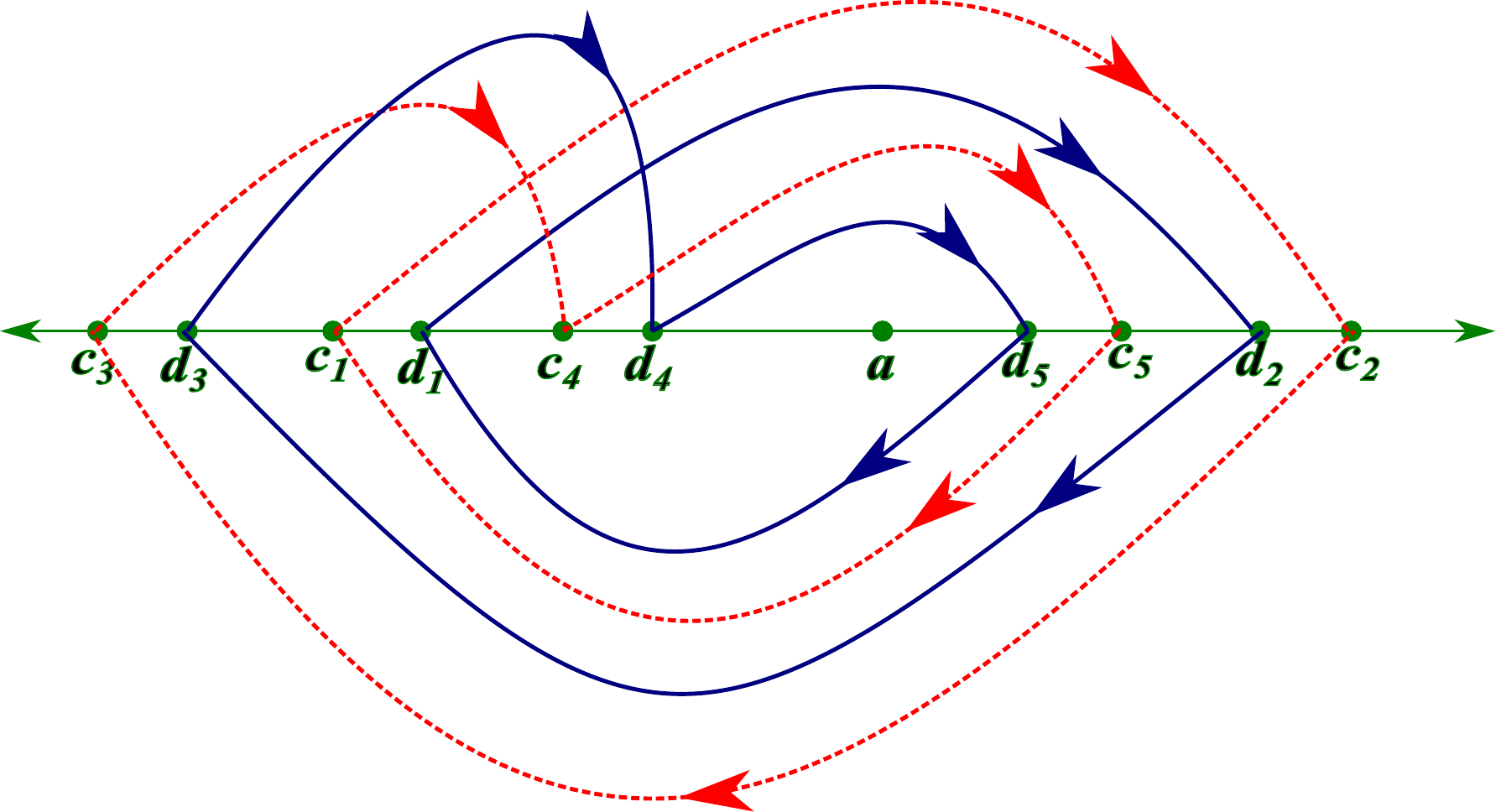}
	\label{fig:double}
\end{figure}

Let us change the map $f$ to obtain a new map $g$ so that the two cycles $P_1$ and $P_2$ are mingled under the map $g$ but $g$ remains unimodal and the code of the composite periodic orbit $g|_\mathcal{P}$ is non-decreasing. For this, we note that here the numerator of our over-rotation number is $p=2$ and $d_1$ is the second point in the orbit $P_2$ counting from the left (here we are using spatial labelling). We define our map $g$ as follows: $g(c_3)=d_1$ (the leftmost point of $P_1$ is mapped to the second point from the left in the spatial labelling of $P_2$). This means that the point $d_5$ which used to be  mapped to $d_1$ under $f$ originally will now have to be mapped somewhere else under $g$. Set $g(d_5)=c_4$. In other words, the point $d_5$ which used to be mapped to $c_1$ has its image shifted by $1$ place to the right. On all other points we keep the actions of the maps $f$ and $g$ exactly the same. In this way, the $f$ orbits $P_1$ and $P_2$ gets amalgamated into one $g$ orbit $\mathcal{P}$.

\begin{figure}[H]
	\caption{\textbf{\textit{Construction of very badly ordered periodic orbit of over-rotation pair $(4, 10)$ obtained from the doubleton lifting }}}
	\centering
	\includegraphics[width=1\textwidth]{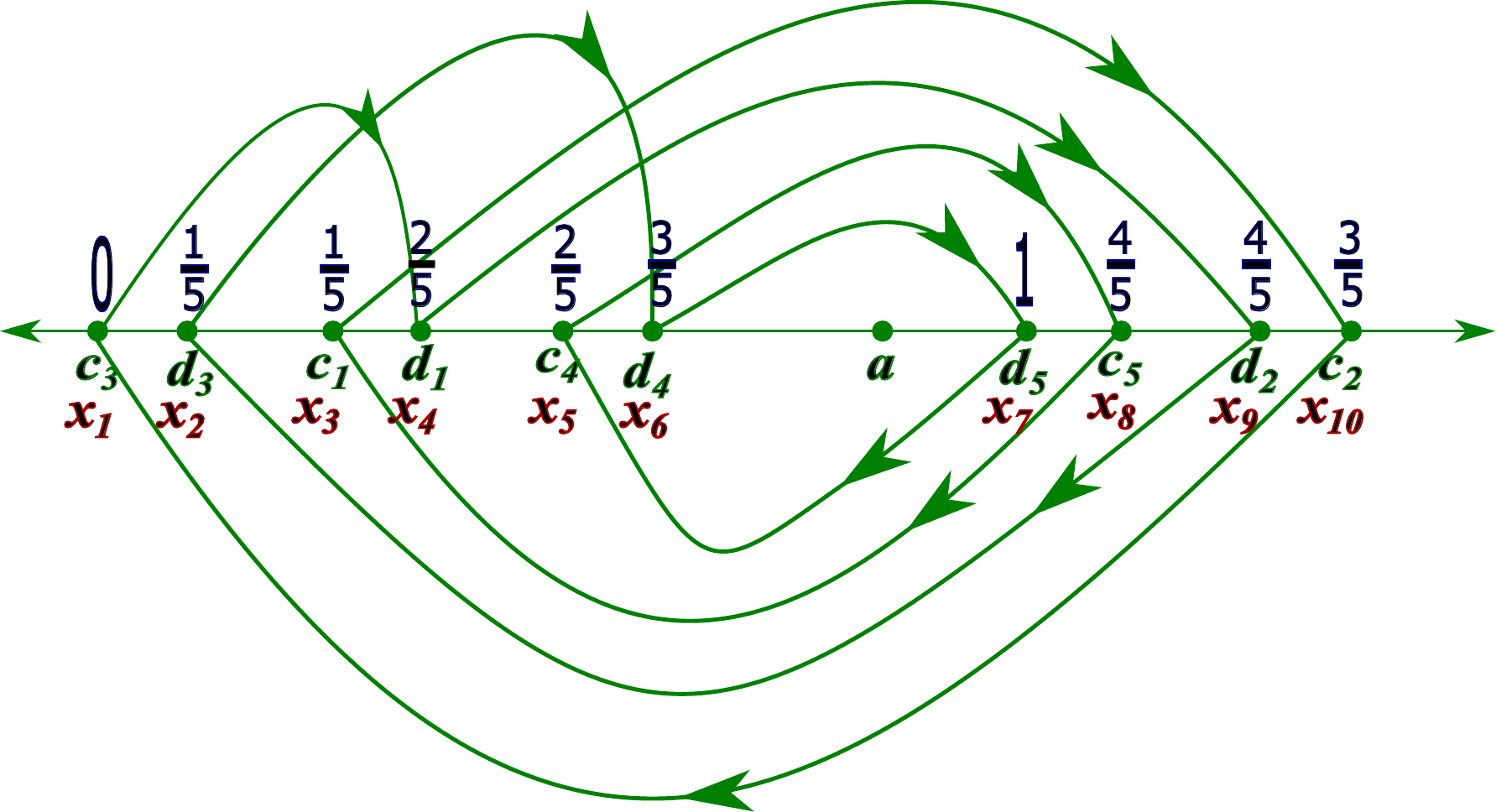}
	\label{exam:badly}
\end{figure}

We now verify that $g|_\mathcal{P}$  is indeed a unimodal very badly ordered periodic orbit of over-rotation number $(4, 10)$. First note that it clearly follows from the construction that $g|_\mathcal{P}$ is unimodal and has over-rotation pair $(4, 10)$. Thus, by Lemma \ref{BM3}, $I_g \supset [\frac{2}{5}, \frac{1}{2}]$. Now, to show that $I_g$ equals $[\frac{2}{5}, \frac{1}{2}]$, we compute out the kneading sequence of $g|_\mathcal{P}$. Let $\nu_g$ be the kneading sequence of $g$. We see that $\nu_g:$ $(R, L, $  $R, R, L,$ $R, R, $ $R, R, C, $ $R, L, $ $R, R, L, $ $R, R,$ $R, R, C, \dots)$  where the fragment $(R, L, $ $R, R, $ $L, R, $ $R, R, $ $R, C)$  is \emph{recurrent}. On the other hand the strongest kneading corresponding to the over-rotation interval $[\frac{2}{5}, \frac{1}{2}]$ is $\nu'_{\frac{2}{5}}=(R, L, $  $R, R, L, R, R, $  $R, R, L, R, R, $  $\dots)$ where there is a \emph{non-repetitive} fragment $(R L) $ in the beginning and then the fragment $ (R, R, L, R, R)$ which is \emph{periodic} and repeats itself again and again. Comparing $\nu_g$ and $\nu'_{\frac{2}{5}}$ we see that they are the same through the symbol number $9$. In the symbol number $10$, we have $C$ in $\nu_g$ and $L$ in  $\nu'_{\frac{2}{5}}$. Observe, that among the first common $9$ symbols of $\nu_g$ and $\nu'_{\frac{2}{5}}$, $R$ occurs $7$ times. Since, $C > L $ in the ordering of symbols, and since $R$ occurs an odd number of times before that, it follows that  $\nu'_{\frac{2}{5}}\succ \nu_g$ as desired. Thus, $\nu'_{\frac{2}{5}}\succ \nu_g \succ \nu_{\frac{2}{5}}$, and it follows that $I_g = [\frac{2}{5}, \frac{1}{2}]$.

Another way to substantiate this fact is by computing out the code of the orbit $g|_\mathcal{P}$. Let us rename the points of $g|_\mathcal{P}$  by spatial labelling as $\{x_1, x_2, \dots x_{10} \}$ (see Figure \ref{exam:badly} on which the new notation for the points is shown in the bottom line of symbols). We set the code of the leftmost point to be equal to $0$, that is,  $L(x_1)=0$, and compute out the codes of the remaining points of the orbit (these codes are shown in the top line of symbols on Figure \ref{exam:badly}). Simple computation gives us : $L(x_2)=\frac{1}{5}$, $L(x_3)=\frac{1}{5}$, $L(x_4)=\frac{2}{5}$,
$L(x_5)=\frac{2}{5}$, $L(x_6)=\frac{3}{5}$, $L(x_7)=1$, $L(x_8)=\frac{4}{5}$,
$L(x_9)=\frac{4}{5}$, $L(x_{10})=\frac{3}{5}$. Notice that the fixed point $a$ of $g$ is located between $d_4$ and $d_5$. Thus, the code for the periodic orbit $g|_\mathcal{P}$ is non-decreasing. Hence, by Lemma \ref{code:nondec},  $I_g = [\frac{2}{5}, \frac{1}{2}]$. This verifies our claim in two different ways.
\end{example}

\end{document}